\documentclass[a4paper,12pt]{article}

\usepackage[margin=2.5cm]{geometry} 
\usepackage{fleqn} 
  
\usepackage{amssymb}
\usepackage{amsmath}
\usepackage{amsfonts}
\usepackage{amsthm}

\usepackage{graphics}
\usepackage{graphicx}
\usepackage{epsfig}
\usepackage{subfigure}
\usepackage{epstopdf}

\usepackage{lmodern}
\usepackage{caption}

\def\r0{\mathcal{R}_0}
\newcommand{\R}{\mathbb{R}}

\newtheorem{prop}{Proposition}
\newtheorem{lemma}{Lemma}
\newtheorem{theorem}{Theorem}[section]
\newtheorem{definition}{Definition}
\newtheorem*{acknowledgements}{Acknowledgements}

\title{Stability switches induced by immune system boosting in an SIRS model with discrete and distributed delays}

\author{M.~V. Barbarossa \thanks{Institute for Applied Mathematics, University of Heidelberg, D-69120 Heidelberg, Im Neuenheimer Feld 205, Germany, \textit{barbarossa\@ uni-heidelberg.de}}
\, and 
M.~Polner \thanks{Bolyai Institute,  University of Szeged, H-6720 Szeged, Aradi v\'{e}rtan\'{u}k tere 1, Hungary,
 \textit{polner\@ math.u-szeged.hu}}\, and
G.~R\"ost \thanks{Bolyai Institute,
 University of Szeged, H-6720 Szeged, Aradi v\'{e}rtan\'{u}k tere 1, Hungary,
 \textit{rost\@ math.u-szeged.hu}}}

\begin{document}

\maketitle

\begin{abstract}
We consider an epidemiological model that includes waning and boosting of immunity. Assuming that repeated exposure to the pathogen fully restores immunity, we derive an SIRS-type model with discrete and distributed delays. First we prove usual results, namely that if the basic reproduction number, $\r0$, is less or equal than $1$, then the disease free equilibrium is globally asymptotically stable, whereas for $\r0>1$ the disease persists in the population. The interesting features of boosting appear with respect to the endemic equilibrium, which can go through multiple stability switches by changing the key model parameters. We construct two-parameter stability charts, showing that increasing the delay can stabilize the positive equilibrium. Increasing $\r0$, the endemic equilibrium can cross two distinct regions of instability, separated by Hopf-bifurcations. Our results show that the dynamics of infectious diseases with boosting of immunity can be more complex than most epidemiological models, and calls for careful mathematical analysis.
\end{abstract}

\section{Motivation and background}
\label{sec:intro}
Classical approaches in mathematical epidemiology present a population divided into susceptibles (S), infectives (I) and recovered (R), and consider interactions and transitions among these compartments. Susceptibles are those hosts who either have not contracted the disease in the past or have lost immunity against the disease-causing pathogen. When a susceptible host gets in contact with an infective one, the pathogen can be transmitted from the infective to the susceptible and with a certain probability, the susceptible host becomes infective himself. After pathogen clearance, that is, when the infective host recovers, a population of memory cells remains in the body. In this way, the host remains immune to the pathogen for a certain time. In case of secondary infection, memory cells respond quickly inducing a boost in the immune system of the host who might show mild or no symptoms. Though persisting for long time after pathogen clearance, memory cells slowly decay, and in the long run recovered hosts could lose pathogen-specific immunity~\cite{Wodarz2007}. \\
\indent Waning immunity is possibly one of the factors which cause, in particular in highly developed regions, recurrent outbreaks of infectious diseases such as chickenpox and pertussis. On the other side, immune system boosting due to contact with infectives prolongs the time during which immune hosts are protected.\\
 \ \\
\noindent A general modeling framework for disease dynamics with waning immunity and immune system boosting in hosts was proposed in~\cite{MVBGR2014JoMB}. The paper introduces a hybrid system of equations in which the immune population is structured by the level of immunity, whereas the susceptible and the infective populations are non-structured. The mathematical model (M2) in~\cite{MVBGR2014JoMB} presents the special case in which immune system boosting restores the maximal immune status, that is, the same immunity level as those induced by natural infection. It was shown in~\cite{MVBGR2014JoMB} that such a special case yields a class of systems with one constant delay ($\tau>0$) and one distributed delay (nonzero only on a finite interval $[-\tau,0]$), where $\tau$ represents the duration of immunity after natural infection. 
In the present work we shall consider a generalization of such systems with constant and distributed delay. The model ingredients and the system of equations are introduced in Section~\ref{sec:model}. Results related to existence, uniqueness and non-negativity of solutions of the delay system are presented in Section~\ref{sec:exiunipos}. Equilibria and criteria for the persistence of the disease are studied in Section~\ref{sec:equilibria}. In Section~\ref{sec:nodelay} we consider the limit case $\tau=0$.  Stability switches and regions of stability in two parameter planes are computed numerically in Section~\ref{sec:stability}.

\section{The model}
\label{sec:model}
Throughout this paper we assume that the birth rate and the natural death rate are equal and constant ($d\geq 0$), and we neglect disease-induced deaths. Hence the total population ($N=S+R+I$) is constant over time and can be normalized, $N(t)\equiv 1$ for all $t\geq 0$.\\
\indent Upon contact with infectives, susceptible hosts contract the disease with transmission rate $\beta I$, $\beta>0$. Infected hosts recover at rate $\gamma>0$, that is, $1/ \gamma$ is the average infection duration. Disease-induced immunity lasts for $\tau>0$ years, after that hosts become susceptible again. Re-exposure to the pathogen boosts the immune system in immune hosts, resetting the clock of the immunity, meaning that hosts who experience immune system boosting are again immune for additional $\tau$ years. A similar assumption was previously proposed in~\cite{Aron1983}. Compared to the model suggested in~\cite{MVBGR2014JoMB}, we include here a generalized boosting force $\nu\geq 0$ as it was previously done in~\cite{Dafilis2012,Grenfell2012}. From a biological point of view, it makes sense to assume $\nu \in [0,1]$, meaning that secondary exposures might have milder effects than primary ones on the immune system. Nevertheless for numerical interest in Section~\ref{sec:stability} we shall consider any $\nu>0$. Under these assumptions, we find two cohorts of individuals entering the susceptible compartment at time $t$ because of immunity loss. On the one side we have hosts who recovered at time $t-\tau$ and since then did not receive immune system boosting nor die,
$$ \gamma I(t-\tau) \exp\left(-d\tau - \nu\beta\int_{-\tau}^{0}I(t+u)\,du\right).$$
On the other side, we have hosts whose immune system was boosted at time $t-\tau$ and who did not die in the time interval $[t-\tau,t]$,
$$ \nu\beta I(t-\tau)R(t-\tau) \exp\left(-d\tau - \nu\beta\int_{-\tau}^{0}I(t+u)\,du\right).$$
All in all, using $R=I-S-I,$ we obtain the system
\begin{equation}
\begin{aligned}
S'(t) & = d(1-S(t)) -\beta I(t)S(t) +I(t-\tau) \left(\gamma + \nu\beta \left(1-S(t-\tau)-I(t-\tau)\right)\right)\\
& \quad \times \exp\left(-d\tau - \nu\beta\int_{-\tau}^{0}I(t+u)\,du\right),\quad t\geq 0,\\
I'(t) & = \beta I(t)S(t) -(\gamma+d) I(t), \quad  t\geq 0, \\[0.3em]
S(t) & =\phi^S(t),  \quad  -\tau \leq t \leq 0,\\
I(t) & =\phi^I(t), \quad -\tau \leq t \leq 0,
\end{aligned} \label{sys:newSISdelay}
\end{equation}
with given initial functions $\phi^S(t) \geq 0,\, \phi^I(t) \geq 0$, such that $\phi^S(t)+\phi^I(t) \leq 1$ for all $t \in [-\tau,0]$. For more details on the derivation of system \eqref{sys:newSISdelay} from a hybrid model with structured immune population we refer to~\cite{MVBGR2014JoMB}. If in system~\eqref{sys:newSISdelay} we neglect population dynamics ($d=0$), assume constant force of infection ($\beta I(t) \equiv h$ for all $t\geq 0$) and set the boosting rate $\nu=1$, then we obtain the system of equations proposed by Aron~\cite{Aron1983}.\\
\ \\
The dynamics of the immune population ($R=1-S-I$) is given by 
\begin{equation*}
R'(t) = -dR(t) +\gamma I(t)
-I(t-\tau) \left(\gamma + \nu\beta R(t-\tau)\right)e^{ -d\tau - \nu\beta\int_{-\tau}^{0}I(t+u)\,du}.
\end{equation*}
Since it does not affect the solutions of~\eqref{sys:newSISdelay}, it can be omitted.

\section{Global existence and uniqueness of solutions to system \eqref{sys:newSISdelay}}
\label{sec:exiunipos}
Let us first introduce some notations from functional differential equations. Let $f:\Omega \to [0,1]^2$, $\Omega \subset C=C([-\tau,0],[0,1]^2) \subset X=C([-\tau,0],\mathbb{R}^2) $.
For all $t\geq 0$, the segment $x_t \in C$ of a function $x(\cdot)$ is defined by $x_t(\theta):=x(t+\theta)$, $\theta \in [-\tau,0]$. The Banach space $X$ is provided with the norm $\|\cdot \|_C$ defined by
 $$\|\phi \|_C\,=\,\sup \left\lbrace |\phi(u)|,\, u \in [-\tau,0] \right\rbrace
\,=\,\sup \left\lbrace |\phi_1(u)|+|\phi_2(u)|,\, u \in [-\tau,0] \right\rbrace.$$
Then the system \eqref{sys:newSISdelay} can be written in the form
\begin{equation}
x_t' = f(x_t),
\label{eq:rfde}
\end{equation}
with $f$ given by
\begin{equation}
\begin{aligned}
f_1(\phi) & = d(1-\phi_1(0)) -\beta \phi_1(0)\phi_2(0)\\
& \quad +\phi_2(-\tau) \left(\gamma + \nu\beta \left(1-\phi_2(-\tau)-\phi_1(-\tau)\right)\right) \times \exp \left( -d\tau - \nu\beta\int_{-\tau}^{0}\phi_2(u)\,du\right),\\[0.3em]
f_2(\phi) & = \beta \phi_1(0)\phi_2(0) -(\gamma+d) \phi_2(0),
\end{aligned} \label{f1_f2}
\end{equation}
where $(\phi_1,\phi_2)=\phi \in \Omega$. Observe that for biological reasons we are interested only in non-negative solutions, hence elements of $\Omega$ are non-negative valued functions. \\
\ \\
\begin{theorem}
\label{thm:exiunipos}
There exists a unique solution to the system \eqref{sys:newSISdelay}, or equivalently, to equation \eqref{eq:rfde} with right-hand side $f:\Omega=C \to [0,1]^2$, defined by \eqref{f1_f2}. Moreover, for this dynamical system the set $\tilde \Omega \subset \Omega$ defined by
\begin{equation*}
\begin{aligned}
\tilde \Omega & =\left\{\phi \in \Omega \qquad \mbox{such that} \right.\\
&  \left.  	1-\phi_1(0)-\phi_2(0)\geq \int_{-\tau}^{0}( \gamma+\nu\beta (1-\phi_1(u)-\phi_2(u)) )\phi_2(u)e^{-\int^{0}_{u}(d+\nu\beta \phi_2(z))\,dz}\,du \right\},
\end{aligned}
 \label{eq:tildeOmega}
\end{equation*}
is positively invariant. 
\end{theorem}
\begin{proof}
\textit{(i) Existence/Uniqueness.} For this result it is sufficient to show that $f$ in \eqref{f1_f2} is Lipschitz continuous in every compact subset $K$ of $\Omega$ \cite[Ch.2]{Kuang1993}. That is, there exists a constant $L>0$ such that, for any $\phi,\psi$ in $K\subset\Omega$ we have
\begin{equation}
\label{eq:lipCond}
\| f(\phi)-f(\psi)\| \leq L\|\phi-\psi\|_C.
\end{equation}
	
First observe that, by definition, $\|\phi\|_C\leq 1$ for any $\phi$ in $K\subset\Omega$. 
We define the auxiliary map $g:\Omega \to  \R$ by 
$$g(\phi)=\exp \left(-\nu\beta\int_{-\tau}^{0}\phi_2(u)\,du\right).$$
Hence, for any $\phi,\,\psi \in \Omega$, it holds that $| g(\phi)|\leq 1$ and 
$|g(\psi)-g(\psi)| \leq \nu \beta \tau \|\phi-\psi\|_C$.

Further, observe that for any $\phi,\,\psi \in \Omega$ we have the estimate
\begin{align*}
| \phi_2(-\tau)g(\phi)-\psi_2(-\tau)g(\psi)|
& \leq |\phi_2(-\tau)g(\phi)-\phi_2(-\tau)g(\psi)|+ | \phi_2(-\tau)g(\psi)-\psi_2(-\tau)g(\psi)|\\[0.1em]
& \leq \| \phi\|_C \,|g(\phi)-g(\psi)|+ | \phi_2(-\tau)-\psi_2(-\tau)|\,|g(\psi)|\\[0.1em]
& \leq(1+\nu \beta \tau )\|\phi-\psi\|_C.
\end{align*}
Then for $\phi,\,\psi \in K\subset\Omega$ we have
\begin{align*}
\|f(\phi)-f(\psi)\|
& = |f_1(\phi)-f_1(\psi)|+|f_2(\phi)-f_2(\psi)|\\[0.3em]
& \leq  d|\phi_1(0)-\psi_1(0)|+(d+\gamma)|\phi_2(0)-\psi_2(0)|\\
& \phantom{=} + 2\beta |\phi_1(0)\phi_2(0)-\psi_1(0)\psi_2(0)|\\
& \phantom{=} +(\gamma+\nu\beta)e^{-d\tau}|\phi_2(-\tau)g(\phi)-
\psi_2(-\tau)g(\psi)|\\[0.3em]
& \phantom{=} +\nu\beta e^{-d\tau}|\phi_2(-\tau)\phi_1(-\tau)g(\phi)-
\psi_2(-\tau)\psi_1(-\tau)g(\psi)|\\[0.3em]
& \phantom{=} +\nu\beta e^{-d\tau}|\phi_2(-\tau)^2g(\phi)-
\psi_2(-\tau)^2g(\psi)|\\[0.3em]
& \leq  d\|\phi-\psi\|_C+ \gamma|\phi_2(0)-\psi_2(0)|\\
& \phantom{=} + 2\beta \left[ |\psi_1(0)||\psi_2(0)-\phi_2(0)|+
|\psi_1(0)-\phi_1(0)||\phi_2(0)|\right]\\ 
& \phantom{=} +(\gamma+\nu\beta)e^{-d\tau}(\nu\beta\tau+1)\|\phi-\psi\|_C\\[0.3em]
& \phantom{=} +\nu\beta e^{-d\tau}\left[|\phi_2(-\tau)||\phi_1(-\tau)|\,|g(\phi)-g(\psi)|\right.\\
& \phantom{=} +|\phi_2(-\tau)|\,|g(\psi)|\,|\phi_1(-\tau)-\psi_1(-\tau)|\\
& \phantom{=} +|\psi_1(-\tau)|\,|g(\psi)|\,|\phi_2(-\tau)-\psi_2(-\tau)|]\\[0.3em]
& \phantom{=} +\nu\beta e^{-d\tau}\left[|\phi_2(-\tau)||\phi_2(-\tau)|\,|g(\phi)-g(\psi)|\,\right.\\
& \phantom{=} +2|\phi_2(-\tau)|\,|g(\psi)|\,|\phi_2(-\tau)-\psi_2(-\tau)|].
\end{align*}
Hence, the estimate in \eqref{eq:lipCond} holds with 
\begin{displaymath}
L\geq d+\gamma+2\beta +e^{-d\tau}\left(\nu \beta +(3\nu\beta+\gamma)(\nu\beta\tau+1)\right).
\end{displaymath}

(ii) \textit{Invariance of $\tilde \Omega$}. 
It is convenient to go back to the explicit formulation of model~\eqref{sys:newSISdelay}. Observe first that the equation for $I$ is an ODE, hence given $I(0)\geq 0$ solutions stay non-negative for all $t\geq 0$. Define the auxiliary function
\begin{displaymath}
A(t)= 1-S(t)-I(t)-\int_{t-\tau}^{t}( \gamma+\nu\beta (1-S(u)-I(u)))I(u)e^{-\int^{t}_{u}(d+\nu\beta I(z))\,dz}\,du.
\end{displaymath}
We remark that for a given solution of system  \eqref{sys:newSISdelay}, $A(t)\geq 0$ is equivalent to $x_t \in \tilde \Omega$, where $x_t$ is the solution of equation~\eqref{eq:rfde}.

Differentiation with respect to $t$ yields
\begin{align*}
A'(t) & = -S'(t)-I'(t)-(\gamma+\nu\beta (1-S(t)-I(t)))I(t)\\
& \phantom{=}+\left(\gamma+\nu\beta \bigl(1-S(t-\tau)-I(t-\tau)\bigr)\right)I(t-\tau)e^{-\int^{t}_{t-\tau}(d+\nu\beta I(z))\,dz}\\
& \phantom{=} \underbrace{-\int_{t-\tau}^{t}( \gamma+\nu\beta (1-S(u)-I(u)))I(u)e^{-\int^{t}_{u}(d+\nu\beta I(z))\,dz}\,du}_{=1-S-I-A}
\left(d+\nu\beta I(t)\right)\\
& = -S'(t)-I'(t)-\gamma I(t)- \nu\beta (1+S(t)+I(t))I(t)\\
& \phantom{=} +(\gamma+\nu\beta (1-S(t-\tau)-I(t-\tau))I(t-\tau)e^{-\int^{t}_{t-\tau}(d+\nu\beta I(z))\,dz}\\
&\phantom{=} +d(1-S(t)-I(t))+\nu\beta I(t)(1-S(t)-I(t)) -A(t)\left(d+\nu\beta I(t)\right).
\end{align*}
Now use \eqref{sys:newSISdelay} and observe that all terms on the right-hand side of the last equation cancel out except for the last one, yielding
$$ A'(t)=-(d+\nu\beta I(t))A(t).$$
Hence for $A(0)\geq 0$, $A(t)$ is nonnegative. This shows that, for given initial data in $\tilde \Omega$, the solution satisfies $S(t)+I(t)\leq 1$ for all $t\geq 0$.\\
\ \\
Finally assume $S\left(\bar t\right)=0$ for some $\bar t>0$ and $S(t)>0$ for $t<\bar t$, as well as $I(t)\geq 0$, for $t\leq \bar t$. Then $\dot S\left(\bar t\right)>0$, as
$$\dot S\left(\bar t\right)=d+ \underbrace{I\left(\bar t-\tau\right) \left(\gamma + \nu\beta \left(1-S\left( \bar t-\tau\right)-I\left(\bar t-\tau\right)\right)\right)
	\exp \left( -d\tau - \nu\beta\int_{-\tau}^{0}I\left(\bar t+u\right)\,du\right)}_{\geq 0},$$  
and the solution $S(t)$ remains non-negative. In particular, given non-negative initial data we have that $S(t)+I(t)\geq 0$ for all $t\geq 0$. This completes the proof.
\end{proof}

\section{Equilibria and persistence} 
\label{sec:equilibria}
In this section we determine equilibria of the system \eqref{sys:newSISdelay} and study their dynamical properties.\\
\ \\
The basic reproduction number of system \eqref{sys:newSISdelay} is
\begin{equation}
	\label{def:r0}
	\r0=\frac{\beta}{d+\gamma}.
\end{equation}
Its value indicates the average number of secondary infections generated in a fully susceptible population by one infected host over the course of his infection. 
The basic reproduction number, $\r0$, is a reference parameter in mathematical epidemiology used to understand if, and in which proportion, the disease will spread among the population.\\
\ \\
Setting the second equation in \eqref{sys:newSISdelay} equal to zero, we see that equilibria satisfy either $I^*=0$ or $S^*=1/\r0$. 
\begin{theorem}
\label{thm:stabilityDFE}
If $\r0\leq1$ there is only one equilibrium, the disease-free equilibrium (shortly: DFE) $(S^*,I^*)=(1,0)$, which is globally asymptotically stable in~$\tilde \Omega$.
\end{theorem}

\begin{proof}
	\textit{(i) Equilibria.} Notice that the $I$ equation can be written as
	$$I'=(d+\gamma)I(t)(\r0 S(t)-1).$$
	Since $S(t)\leq 1$, for an equilibrium we either have $I=0$ or $\r0=1$ and $S=1$. Both possibilities retain the DFE.\\
	\ \\	
	\textit{(ii) Convergence.} Clearly $I(t)$ is decreasing and bounded from below by zero thus converges. Assuming $\lim_{t\to \infty}I(t)= q>0$, it holds that there is a $t_q$ such that for $t>t_q$ we have $I(t)>q/2$ .
	Consequently $S(t)<1-q/2<1$ for $t>t_q$. But then for $t>t_q$,
	$$I(t)'<(d+\gamma)I(t)(-q/2)$$ which implies $\lim_{t\to \infty}I(t)= 0$.  Let
	$$\eta(t):= I(t-\tau) \left(\gamma + \nu\beta \left(1-S(t-\tau)-I(t-\tau)\right)\right) \exp\left(-d\tau - \nu\beta\int_{-\tau}^{0}I(t+u)\,du\right),$$
	Then we may write
	$$S'(t)=d(1-S(t))-\beta I(t)S(t)+\eta(t),$$ 
	where $\lim_{t \to \infty}I(t) =0$ and $\lim_{t \to \infty}\eta(t) =0$, from which we can easily deduce $\lim_{t \to \infty} S(t)= 1$.\\
	\ \\	
	\textit{(iii) Stability.}
	For any $\epsilon>0$, let us choose $\delta:=\max\{\epsilon,\frac{\beta \epsilon}{d}\}$. We claim that if the initial condition is in the $\epsilon$-neighborhood of the DFE, then the solution stays in the $\delta$-neighborhood of the DFE. Let $(\phi,\psi) \in \tilde \Omega$ be initial conditions such that $\phi(s)>1-\epsilon$ and $\psi(s)<\epsilon$. Then $I(t)<\epsilon$ for all $t>0$ since $I(t)$ is decreasing, and the inequality
	$S'(t)>d(1-S(t))-\beta\epsilon$ holds for $t\geq 0$. Consider the comparison equation $y(t)'=d(1-y(t))-\beta \epsilon$ with $y(0)=S(0)>1-\epsilon$. Since $y(t)$ converges to $1-\frac{\beta \epsilon}{d}$ monotonically and $S(t)\geq y(t)$, we either have $S(t)>1-\epsilon $ (if $S(0)<1-\frac{\beta \epsilon}{d}$)  or $S(t)\geq1-\frac{\beta \epsilon}{d}$ (if $S(0)\geq 1-\frac{\beta \epsilon}{d}$) for all $t\geq 0$, and the solution remains in the $\delta-$neighborhood of the DFE. 
\end{proof}
\begin{prop}
For $\r0>1$ there is a unique endemic equilibrium $(S^*,I^*)=(1/\r0,I^*)$, with $I^*>0$.
\end{prop}
\begin{proof}
Assume $\r0>1$. Endemic equilibria are given by positive intersection points, $x=I^*$, of the line $y_1(x)$ and the curve $y_2(x)$, where
\begin{align}
y_1(x)& = (\gamma+d) x -d\left(1-\frac{1}{\r0}\right), \label{def:line_y1}\\
y_2(x) & = \alpha x \left(\rho-\kappa x\right)e^{-\eta x},\qquad 0\leq x\leq \rho/ \kappa, \label{eq:study_curve_Istar_maxboostcase}
\end{align}
with the coefficients defined as
\begin{align*}
\kappa &:=\nu \beta >0,\\
\rho &:= \gamma+ \kappa \left(1-\frac{1}{\r0}\right)>0, \mbox{ as } \r0>1,\\
\alpha&:=e^{-d\tau}>0,\\
\eta& :=\kappa\tau>0.
\end{align*}
As $\gamma,\,d>0$ and $\r0>1$, the line $y_1(x)$ has a negative y-intercept,  $\left(0,-d\left(1-\frac{1}{\r0}\right)\right)$ and a positive $x$-intercept, $\left(\frac{d}{d+\gamma}\left(1-\frac{1}{\r0}\right),0\right)$. It is obvious that $y_2(x)\geq 0$ for all $x\geq 0$, $y_2(0)=0 =y_2(\rho/\kappa)$. The first derivative of \eqref{eq:study_curve_Istar_maxboostcase} is given by
\begin{align*}
y_2'(x) & =\alpha \left(\eta\kappa x^2-x(2\kappa+\eta\rho)+\rho\right)e^{-\eta x}.
\end{align*}
It follows that $y_2'(x)=0$ for
$$x_{1,2}=\frac{\eta\rho +2\kappa \pm \sqrt{\eta^2\rho^2+4\kappa^2}}{2\eta\kappa}.$$
Observe that
\begin{displaymath}
x_{2}  =\frac{\eta\rho +2\kappa+ \sqrt{\eta^2\rho^2+4\kappa^2}}{2\eta\kappa}\; >\; \frac{\eta\rho +2\kappa+ \sqrt{(\eta\rho-2\kappa)^2}}{2\eta\kappa}\; = \;\frac{\rho}{\kappa},
\end{displaymath}
hence we have only one extremal point in the definition interval $[0,\rho/\kappa]$. The extremal point $(x_1,y(x_1))$ is a local maximum as $y_2(0)=0$, $y_2'(0)>0$ and $\lim y_2(x)=0$ for $x \to \rho/\kappa$.\\
\indent To guarantee the existence of a unique intersection point between the line $y_1$ and the curve $y_2$, we  determine the inflection points. We compute the second derivative 
\begin{equation*}
y_2''(x)  =\alpha \left(-\eta^2\kappa x^2+x(4\eta\kappa+\eta^2\rho)-2(\rho\eta+\kappa) \right)e^{-\eta x}.
\end{equation*}
The inflection points of $y_2(x)$ are 
$$x_{a,b}=\frac{\eta\rho +4\kappa \pm \sqrt{\eta^2\rho^2+8\kappa^2}}{2\eta\kappa}.$$
Observe that 
$$ (\eta\rho-2\kappa)^2 \leq\eta^2\rho^2+8\kappa^2 \leq (\eta\rho+4\kappa)^2.$$
Hence we have
\begin{displaymath}
x_{a}  = \frac{\eta\rho +4\kappa - \sqrt{\eta^2\rho^2+8\kappa^2}}{2\eta\kappa}\;> \;\frac{\eta\rho+ 4\kappa - \sqrt{(\eta\rho+4\kappa)^2}}{2\eta\kappa} = 0,
\end{displaymath} 
 whereas $x_b\geq x_a$ and
\begin{displaymath}
x_{b}  = \frac{\eta\rho +4\kappa + \sqrt{\eta^2\rho^2+8\kappa^2}}{2\eta\kappa}\; \geq
 \; \frac{\eta\rho+ 4\kappa + \sqrt{(\eta\rho-2\kappa)^2}}{2\eta\kappa}\; \geq\;  \frac{\rho}{\kappa}+\frac{1}{\eta} >\frac{\rho}{\kappa}.
\end{displaymath}
Hence the point $x_b\notin [0,\rho/\kappa]$ and the function $y_2(x)$ has at most one inflection point in its domain.
It follows that there is only one intersection point $x=I^*>0$, which corresponds to the endemic equilibrium.
\end{proof} 
Next we prove the persistence of the disease for $\r0>1$. 
Consider the semiflow $\Phi$ on  $\tilde \Omega$, defined by the unique global solutions. 
Let us define the persistence function
\begin{equation*}
\rho:\tilde \Omega \to \mathbb{R}_{+}, \rho(\phi)=\phi_2(0).
\end{equation*}
Let
\begin{align*}
\tilde \Omega_{+}&:=\{\phi\in \tilde \Omega |\rho(\phi)>0\},\\
\tilde \Omega_{0}&:=\{\phi\in \tilde \Omega |\rho(\phi)=0\}=\tilde \Omega \setminus{\tilde \Omega_{+}},
\end{align*}
where $\tilde \Omega_{0}$ is called the extinction space corresponding to $\rho$, because $\tilde \Omega_{0}$ is the collection of states where the disease is not present. 
\begin{prop}
\label{pro:tilde}
The sets  $\tilde \Omega_{0}$ and $\tilde \Omega_{+}$ are forward invariant under the semiflow~$\Phi$.
\end{prop}
\begin{proof}
From (2.1) we have $I(t)=I(t_0)e^{\int_{t_0}^t (\beta S(u)-\gamma-d)du}$ for any $t\geq t_0$, hence  $I(t_0)=0$ implies $I(t)=0$  and $I(t_0)>0$ implies $I(t)>0$ for all $t\geq t_0$.
\end{proof}
We now introduce some terminology of persistence theory from \cite[Chapters 3.1 and 8.3]{ThiemeSmith}.
\begin{definition}
Let $X$ be a nonempty set and $\rho: X\rightarrow \mathbb{R}_{+}$.
\begin{enumerate}
\item A semiflow $\Phi :\mathbb{R}_{+}\times X \rightarrow X$ is called uniformly weakly $\rho$-persistent, if there exists some $\epsilon>0$ such that
\begin{equation*}
\limsup_{t\to\infty} \rho (\Phi (t,x))>\epsilon \qquad \forall \,x\in X,~ \rho(x)>0.
\end{equation*}
\item A semiflow $\Phi$ is called uniformly (strongly) $\rho$-persistent, if there exists some $\epsilon>0$ such that
\begin{equation*}
\liminf_{t\to\infty} \rho (\Phi (t,x))>\epsilon \qquad \forall \,x\in X,~ \rho(x)>0.
\end{equation*}
\item A set $M\subset X$ is called weakly $\rho$-repelling if there is no $x\in X$ such that $\rho(x)>0$ and $\Phi(t,x)\to M$ as $t\to\infty$.
\end{enumerate}
\end{definition}
For a function $f:\mathbb{R_+}\to\mathbb{R}$, we use the notation
\begin{equation*}
f^{\infty}=\limsup_{t\to\infty}f(t) \qquad \text{and}\qquad 
f_{\infty}=\liminf_{t\to\infty}f(t).
\end{equation*}
\begin{theorem}
\label{thm:rho}
If $R_{0}>1$, then the semiflow $\Phi$ is uniformly $\rho$-persistent,  i.e. there is a $\delta>0$ such
that for any solution $\liminf_{t \to \infty} I(t)\geq \delta$.
\end{theorem}
\begin{proof}
We shall apply Theorem~4.5 and Theorem~8.17 in~\cite{ThiemeSmith}.
First we show that the disease free equilibrium $(1,0)$ 
is weakly $\rho$-repelling. Suppose that there exists $\psi_{0}\in \tilde \Omega$ such that $\rho (\psi_{0})>0$ with
\begin{equation}\label{eq:DFEattract}
\lim_{t\to\infty}\Phi (t,\psi_{0})=(1,0). 
\end{equation}
For such a solution, $I(0)>0$ and $\lim_{t \to \infty} I(t)=0$. Then for any $\epsilon>0$ there is a $T>0$ such that $I(t)<\epsilon/\beta$ for all $t>T$. Then, for $t>T$, we have the relation
$$S'(t)\geq d(1-S(t))-\epsilon,$$
hence $S_\infty\geq 1-\epsilon/d $. This holds for any $\epsilon>0$, so for sufficiently large $t$,
$S(t)\r0>1$ holds, and from (2.1) we have
$$I'(t)=I(t)(d+\gamma)(\r0 S(t)-1)>0,$$
that contradicts $I(t)\to 0$. Thus, there is no $\psi_0 \in \tilde \Omega$ such that $\rho(\psi_{0})>0$ and \eqref{eq:DFEattract} holds and $(1,0)$ 
is weakly $\rho$-repelling.
By Proposition \ref{pro:tilde}, together with the fact that whenever $I(t)\equiv 0$, $S'(t) =d(1-S(t))$ so $S(t)\to 1$, and thus $\cup_{\phi\in \tilde \Omega_{0}}\omega(\phi)=\{(1,0)\}$, one can see that $\Phi$ is uniformly weakly $\rho$-persistent using  Theorem~8.17 in \cite{ThiemeSmith}. Since $\Phi$ has a compact global attractor on $\tilde \Omega$, we can apply Theorem~4.5 in~\cite{ThiemeSmith} to conclude that $\Phi$ is uniformly $\rho$-persistent.
\end{proof}

\section{The case $\tau=0$}
\label{sec:nodelay}
Assume that there is no disease-induced immunity, that is, hosts become susceptible immediately after recovery.
Setting $\tau=0$ in \eqref{sys:newSISdelay} and omitting the time argument, we obtain
\begin{equation}
\begin{aligned}
S' & = d(1-S) -\beta IS+I \left(\gamma + \nu\beta \left(1-S-I\right)\right),\\[0.3em]
I' & = \beta IS -(\gamma+d) I.
\end{aligned} \label{sys:tau0}
\end{equation} 
Observe that if $\nu=0$, this is a standard SIS model, cf. \cite{Brauer2001}.
\begin{prop}
	\label{prop:tau0_invariantset}
	The set 
	\begin{displaymath}
	\mathcal{D}=\left\{(S,I) \in \R^2, \; S\geq 0, \, I\geq 0,\, S+I\leq 1\right\}
	\end{displaymath}
	is positively invariant for the system \eqref{sys:tau0}.
\end{prop}
\begin{proof}
	Assume $S\left(\bar t\right)=0$, and $0\leq I(\bar t )\leq 1$ for some $\bar t>0$. Then 
	$$S'\left(\bar t\right) = d+I\left(\bar t\right)\left(\gamma + \nu\beta \left(1-I\left(\bar t\right)\right)\right)\geq 0,$$
	hence the $S$-component remains non-negative. Similarly,
	if $I\left(\bar t\right)=0$ for some $\bar t >0$, then $I'\left(\bar t\right)=0$.
	Now consider the sum $S+I$,
	$$
	0\leq S+I\leq 1 \qquad \Rightarrow \qquad (S+I)' = (1-(S+I))(d+\nu\beta I)\geq 0.$$
	If $(S+I)\left(\bar t\right)= 1$, then $(S+I)'\left(\bar t\right) = 0$.\\
	Hence, any solution starting in $\mathcal D$ remains in this set.
\end{proof}
Observe that the  limit case, system \eqref{sys:tau0}, has the same basic reproduction number in \eqref{def:r0} as the system with delay \eqref{sys:newSISdelay}.
\begin{prop}
	\label{prop:tau0_stability}
If $\r0\leq 1$, the disease-free equilibrium is the only equilibrium of \eqref{sys:tau0} and it is globally asymptotically stable.
If $\r0>1$, there is a unique endemic equilibrium which is globally asymptotically stable in $\mathcal D \setminus \mathcal D_0$, where $	\mathcal{D}_0=\left\{(S,0) \in \R^2, \; 1 \geq S\geq 0 \right\}$ . In this case the DFE is unstable, but attracts solutions in the invariant disease free subspace $\mathcal D_0$.
\end{prop}
\begin{proof}
	(i) \textit{Existence/Uniqueness of equilibria.} It is trivial to observe that the system \eqref{sys:tau0} has a unique DFE, namely, $(S^*,I^*)=(1,0)$.\\
	\indent Consider endemic equilibria $(S^*,I^*)$, with $I^*>0$. From the second equation in \eqref{sys:tau0}, we see that $S^*=1/\r0>0$, as $\r0>0$ and $I^*$ solves the quadratic equation 
	\begin{equation}
	\nu \beta (I^*)^2 -I^*\left( \nu\beta \left(1-\frac{1}{\r0}\right)-d\right)+d\left(\frac{1}{\r0}-1\right)=0.
	\label{eq:Istarimplict}
	\end{equation}
	The proof of uniqueness of the endemic equilibrium is given by the following observation. For $\r0>1$, the graph of \eqref{eq:Istarimplict} is a parabola opening upwards with negative y-intercept. Hence there is a unique strictly positive zero $I^*$ of \eqref{eq:Istarimplict}. For $\r0<1$, this parabola has a positive y-intercept and its vertex is on the negative half-plane, hence \eqref{eq:Istarimplict} has no positive zeroes. For $\r0=1$, the only biologically relevant solution is $I^*=0$.	Observe that
	\begin{displaymath}
	(S+I)' = \left(d+\nu\beta I \right)(1-(S+I))\qquad \Rightarrow \qquad (S+I)' =0 \Leftrightarrow S^*+I^*=1.
	\end{displaymath}
	In particular, for $\r0>1$ we have the relation $I^*=1-1/\r0>0$.\\
	\ \\
	(ii) \textit{Linearized stability.} 	
	The Jacobian matrix of the ODE system  \eqref{sys:tau0} is 
	\begin{equation}
	J(S,I)=\begin{pmatrix}
	-d-\beta I(1+\nu) & \gamma-\beta S+\nu \beta(1-S-2I)\\
	\beta I & \beta S-\gamma -d
	\end{pmatrix}.
	\label{eq:jacmax_ODE2d}
	\end{equation}
	Evaluation of \eqref{eq:jacmax_ODE2d} at the DFE yields an upper-triangular matrix with eigenvalues $\lambda_1=-d<0$ and $\lambda_2= \beta(1-1/\r0)$. Hence, if $\r0<1$ the DFE is locally asymptotically stable, whereas if $\r0>1$, the DFE is unstable. For $\r0>1$ we evaluate the Jacobian matrix \eqref{eq:jacmax_ODE2d} at the endemic equilibrium
	\begin{equation*}
	J_E=J(1/\r0,I^*)=\begin{pmatrix}
	-d-\beta I^*(1+\nu) & -d-\nu \beta I^* \\
	\beta I^* & 0
	\end{pmatrix},
	\end{equation*}
	with $I^*=1-1/\r0$. We see that $\mbox{Tr}(J_E)=-d-\beta I^*(1+\nu)<0$, and $det(J_E)  = \beta I^*\left(d+\nu \beta I^*\right)>0$, hence the endemic equilibrium is locally asymptotically stable.\\
	\ \\
	(iii)\textit{Global stability.}
	If $\r0\leq1$, the global stability of the DFE can be proved exactly the same way as we proved it for the more general case in Theorem~\ref{thm:stabilityDFE}. Similarly, for $\r0>1$, the same argument as in Theorem~\ref{thm:rho} shows that for any solution with $I(t)>0$, we have $\liminf_{t\to \infty} I(t)>0$. To avoid repetition, here we only prove the global stability of the endemic equilibrium when $\r0>1$ and $\tau=0$.
	We can exclude the existence of periodic orbits using the negative criterion of Bendixson-Dulac \cite{GuckenheimerHolmes}. Using Proposition~\ref{prop:tau0_invariantset}, it is clear that $\mathcal{D}_0$ and $\mathcal D \setminus \mathcal D_0$ are both invariant, and in $\mathcal D_0$ solutions tend to the DFE. In $\mathcal D \setminus \mathcal D_0$, dividing the system \eqref{sys:tau0} by $I$, we obtain
	\begin{align*}
	\frac{1}{I }S' & = \frac{d}{I}(1-S) -\beta S+\left(\gamma + \nu\beta \left(1-S-I\right)\right)\,=:f(S,I),\\[0.3em]
	\frac{1}{I}  I' & = \beta S -(\gamma+d)\,=:g(S,I).
	\end{align*}
	Then we compute the divergence of $f$ and $g$:
	\begin{displaymath}
	\mbox{div}(f,g)=\frac{\partial f}{\partial S}+\frac{\partial g}{\partial I}
	= -\frac{d}{I} -\beta -\nu\beta<0 \quad \mbox{for all}\quad (S,I)\in \mathcal D \setminus \mathcal D_0.
	\end{displaymath}
	The divergence remains negative for all $(S,I)\in \mathcal D \setminus \mathcal D_0$, hence we can exclude existence of periodic orbits in $\mathcal D \setminus \mathcal D_0$. Using this result together with the theorem of Poincar\'e-Bendixson, we have that all trajectories of the two dimensional ODE system \eqref{sys:tau0} in $\mathcal D \setminus \mathcal D_0$ converge to an equilibrium. Since $\liminf_{t\to \infty} I(t)>0$, it has to be the stable endemic equilibrium.
\end{proof}

\section{Stability}
\label{sec:stability}
In this section we consider linearized stability properties of the system \eqref{sys:newSISdelay}, considering first stability switches with respect to $\tau$ and then with respect to more parameters.

\subsection{Associated linear system}
We linearize the system \eqref{sys:newSISdelay} about an equilibrium point $(S^*,I^*)$, with $I^*\geq 0$, and introduce the variables $x(t)$ and $y(t)$ such that
$S(t)=S^*+x(t),\, I(t)=I^*+y(t)$. The conditions at equilibrium are given by
\begin{align}
0 & = d(1-S^*) -\beta I^*S^* +I^* \left(\gamma + \nu\beta \left(1-S^*-I^*\right)\right)e^{-\tau \left(d + \nu\beta I^*\right)}, \label{condi_equi1} \\[0.3em]
0 & = \beta I^*S^* -(\gamma+d) I^*.
\label{condi_equi2}
\end{align} 
Using the condition \eqref{condi_equi2}, we linearize the equation for $I$ in  \eqref{sys:newSISdelay} and obtain
\begin{equation*}
\dot y(t) 
= \beta I^*x(t)+(\beta S^*-\gamma-d)y(t).
 \label{lin_syst_equil_y} 
\end{equation*}
Linearization of the $S$-equation is less trivial. First we consider the exponential term
\begin{align*}
\exp \left( -d\tau - \nu\beta\int_{t-\tau}^{t}I(u)\,du\right) 
= e^{-\tau(d+ \nu\beta I^*)} \left(1- \nu\beta\int_{t-\tau}^{t} y(u)\,du\right).
\end{align*}
Some computations and the condition at equilibrium \eqref{condi_equi1} yield
\begin{align*}
\dot x(t) & = -\beta S^*y(t)-(d+ \beta I^*)x(t)\\
& \quad  -\nu\beta I^*e^{-\tau(d+ \nu\beta I^*)}x(t-\tau)\\
& \quad + \bigl( \gamma + \nu\beta (1-S^*-2 I^*)\bigr)e^{-\tau(d+ \nu\beta I^*)}y(t-\tau)\\
& \quad - \nu\beta I^* \bigl(\gamma + \nu\beta (1 -S^*-I^*)\bigr) e^{-\tau(d+ \nu\beta I^*)}\int_{t-\tau}^{t} y(u)\,du.
\end{align*}
Now we use the Ansatz $x(t)=x_0e^{\lambda t}$ and $y(t)=y_0e^{\lambda t}$, with $(x_0,y_0)\neq (0,0)$, 
\begin{align}
x_0\lambda e^{\lambda t} & = -\beta S^*y_0e^{\lambda t}-(d+ \beta I^*)x_0e^{\lambda t} \nonumber\\
& \quad  -\nu\beta I^*e^{-\tau(d+ \nu\beta I^*)}x_0e^{\lambda t}e^{-\lambda \tau}\nonumber\\
& \quad + y_0e^{\lambda t}e^{-\lambda \tau} \bigl( \gamma + \nu\beta (1-S^*-2 I^*)\bigr)e^{-\tau(d+ \nu\beta I^*)} \label{eq:char_form1x}\\
& \quad - I^* \bigl(\gamma + \nu\beta (1 -S^*-I^*)\bigr) e^{-\tau(d+ \nu\beta I^*)}
\left(\nu\beta\int_{0}^{\tau} y_0 e^{\lambda u}e^{\lambda t}e^{-\lambda \tau}\,du\right),\nonumber\\
y_0\lambda e^{\lambda t}  & = \beta I^*x_0e^{\lambda t}+(\beta S^*-\gamma-d)y_0e^{\lambda t}. \label{eq:char_form1y}
\end{align}\pagebreak 

\noindent The next statement will play an important role in our stability analysis.
\begin{lemma}
	\label{lemma:lambda0}
$\lambda=0$ is not a root of \eqref{eq:char_form1x}--\eqref{eq:char_form1y}.
\end{lemma}
\begin{proof}
Consider first the DFE, $(S^*,I^*)=(1,0)$, 
and set $\lambda=0$ into \eqref{eq:char_form1x}--\eqref{eq:char_form1y}. Then the system reduces to
\begin{align*}
0 &= -\beta y_0-dx_0+y_0\gamma e^{-\tau d},\\
0 &= (\beta-\gamma-d)y_0.
\end{align*}
The last equation implies $y_0=0$, as in general $\beta\neq\gamma+d$. It follows that also $x_0=0$. But this contradicts the existence of a nontrivial solution of the linear system.\\
\ \\
Next we consider the endemic equilibrium ($S^*, I^*$), where $S^*=1/\r0=(\gamma+d)/\beta$ and $I^*>0$, and set $\lambda=0$ into \eqref{eq:char_form1y}. We obtain 
$$ 0  = \beta I^*x_0 \quad \Rightarrow \quad x_0=0.$$ 
Substitute $x_0=0$ in \eqref{eq:char_form1x} and find
\begin{align*}
0& = y_0\left(-\gamma-d + \bigl( \left(\gamma + \nu\beta (1-1/\r0- I^*)\right)(1-\nu\beta\tau I^*)  -\nu\beta I^* \bigr)e^{-\tau(d+ \nu\beta I^*)}\right).
\end{align*}
This can be written as
\begin{equation}
\label{eq:Ftau2}
0 = - y_0e^{-\tau(d+ \nu\beta I^*)}F(\tau),
\end{equation}
with $F:\R\to\R$ defined by
\begin{equation*}
\label{eq:F_tau}
F(\tau)= (\gamma+d)e^{\tau(d+ \nu\beta I^*)} + \left(\gamma + \nu\beta (1-1/\r0- I^*)\right)(\nu\beta\tau I^*-1)  +\nu\beta I^*.
\end{equation*}
This function satisfies
$$F(0)= d+\nu\beta I^*>0,$$ 
and 
$$\frac{\partial F}{\partial \tau}=(\gamma+d)(d+\nu\beta I^*)e^{\tau(d+ \nu\beta I^*)}+\nu\beta I^*\left(\gamma+\nu\beta (1-1/\r0-I^*)\right)>0,\quad \forall\tau\geq 0.$$
Hence $F$ is a strictly increasing function of $\tau.$ Therefore, $F(\tau)>0$ for all $\tau\geq 0,$ which implies 
that (\ref{eq:Ftau2}) holds only when $y_0=0,$ which is a contradiction to the existence of nontrivial solutions of the linear system. 
\end{proof}
Observe that
$$ \int_{0}^{\tau} e^{\lambda u}\,du 
= \left[\frac{e^{\lambda u}}{\lambda} \right]^{u=\tau}_{u=0} = \frac{e^{\lambda \tau}-1}{\lambda}, \qquad \mbox{for} \quad \lambda \neq 0.$$
Substitute this expression into \eqref{eq:char_form1x}--\eqref{eq:char_form1y}, and divide both equations by $e^{\lambda t}$. This yields
\begin{equation}
\begin{aligned}
x_0\lambda & = -\beta S^*y_0-(d+ \beta I^*)x_0\\
& \quad  -\nu\beta I^*e^{-\tau(d+ \nu\beta I^*)}x_0 e^{-\lambda \tau}\\
& \quad + y_0e^{-\lambda \tau} \bigl( \gamma + \nu\beta (1-S^*-2 I^*)\bigr)e^{-\tau(d+ \nu\beta I^*)}\\
& \quad - I^* \bigl(\gamma + \nu\beta (1 -S^*-I^*)\bigr) e^{-\tau(d+ \nu\beta I^*)}
\nu\beta y_0 \frac{1-e^{-\lambda \tau}}{\lambda},
\end{aligned}
\label{eq:char_x_1}
\end{equation}
respectively,
\begin{equation}
y_0\lambda  = \beta I^*x_0+(\beta S^*-\gamma-d)y_0.
\label{eq:char_y_1}
\end{equation}
\ \\
Multiply equation \eqref{eq:char_x_1} by $\lambda\neq 0$ and obtain
\begin{equation}
\begin{aligned}
x_0\lambda^2 & = -\lambda\left(\beta S^*y_0+(d+ \beta I^*)x_0\right)\\
& \quad  - \lambda e^{-\lambda \tau} \nu\beta I^*e^{-\tau(d+ \nu\beta I^*)}x_0\\
& \quad + y_0\lambda e^{-\lambda \tau} \bigl( \gamma + \nu\beta (1-S^*-2 I^*)\bigr)e^{-\tau(d+ \nu\beta I^*)}\\
& \quad - I^* \bigl(\gamma + \nu\beta (1 -S^*-I^*)\bigr) e^{-\tau(d+ \nu\beta I^*)}
\nu\beta y_0 \left(1-e^{-\lambda \tau}\right).
\end{aligned}
\label{eq:char_x_2}
\end{equation}

From equation \eqref{eq:char_y_1} we have that 
\begin{equation}
\frac{x_0}{y_0} = \frac{\lambda +d+\gamma- \beta S^*}{\beta I^*} = \frac{\lambda}{\beta I^*}+\frac{d+\gamma-\beta S^* }{\beta I^*},\label{eq:char_y_2}
\end{equation}
where $\lambda \neq 0$ is the solution of the characteristic equation determined by \eqref{eq:char_x_2}. 
Divide by $y_0$ and multiply by $\beta I^*$ equation \eqref{eq:char_x_2}, and substitute the expression \eqref{eq:char_y_2}. In this way we obtain 
\begin{equation}
\begin{aligned}
\lambda^2 \left(\lambda -\beta S^* +d+\gamma \right)
& = -\lambda \beta^2 I^* S^*-\lambda (d+ \beta I^*)(\lambda + d+\gamma-\beta S^*)\\
& \quad  - \lambda e^{-\lambda \tau} \nu \beta I^* e^{-\tau(d+ \nu\beta I^*)}(\lambda + d+\gamma-\beta S^*) \\
& \quad + \lambda e^{-\lambda \tau}  \beta I^* \bigl( \gamma + \nu\beta (1-S^*-2 I^*)\bigr)e^{-\tau(d+ \nu\beta I^*)}\\
& \quad - \nu(\beta I^*)^2  \bigl(\gamma + \nu\beta (1 -S^*-I^*)\bigr) e^{-\tau(d+ \nu\beta I^*)}\\
& \quad + e^{-\lambda \tau} \nu(\beta I^*)^2  \bigl(\gamma + \nu\beta (1 -S^*-I^*)\bigr) e^{-\tau(d+ \nu\beta I^*)}.
\end{aligned}
\label{eq:char_x_4}
\end{equation}
This is the characteristic equation we get from linearization about a generic stationary point $(S^*,I^*)$. There is no need to discuss further the DFE, as we already know it is globally asymptotically stable~(Theorem~\ref{thm:stabilityDFE}). Therefore we shall consider only the endemic equilibrium.

\subsection{Stability switches with respect to $\tau$}
The characteristic equation and its roots are functions of the delay $\tau$. The stability of an equilibrium solution (in the following, the endemic equilibrium) may change as the length of the delay changes \cite{Kuang1993,DiekmannRFDE}.\\
\ \\
The characteristic equation about the endemic equilibrium is
\begin{equation}
\begin{aligned}
\lambda^3  & = -\lambda \beta(\gamma+d) I^* -\lambda^2 (d+ \beta I^*) - \lambda^2 e^{-\lambda \tau} \nu \mu(\nu)\\
& \quad + \lambda e^{-\lambda \tau} \mu(\nu) \bigl( \sigma(\nu)  - \nu\beta I^*\bigr)
- \nu \beta I^* \sigma(\nu) \mu(\nu) + e^{-\lambda \tau} \nu\beta I^*  \sigma(\nu) \mu(\nu),
\end{aligned}
\label{eq:char_x_5}
\end{equation}
with
\begin{equation}
\begin{aligned} 
\mu(\nu) & = \beta I^* e^{-\tau(d+ \nu\beta I^*)},\\
\sigma(\nu) &=\gamma + \nu\beta (1 -1/\r0-I^*).
\end{aligned}
\label{def:mu_sigma}
\end{equation}
The next proposition shows that characteristic roots (as continuous functions of $\tau$) are bounded on the right half plane. Hence, Rouch\'e's theorem \cite{Dieudonne1960} implies that roots $\lambda(\tau)$ cannot suddenly appear or disappear, nor they can change multiplicity at a finite point in the complex plane.\\
\begin{prop}[Bounded characteristic roots]
	If $\Re(\lambda)>0$, then
	 $$|\lambda|\leq \max \left\{1, (\hat a+\hat b+\hat c)\right\},$$ 
	 where $\hat a=d+\beta(\nu+1)$, $\hat b=\beta(\gamma+\beta(2\nu+1))$, and $\hat c=2\nu\beta^2(\gamma+\nu\beta)$.
\end{prop}
\begin{proof}
Assume $\Re(\lambda)>0$. If $|\lambda|> 1$, from \eqref{eq:char_x_5} we have the estimate (recall that $I^*\leq 1$)
\begin{align*}
|\lambda|^3  & \leq  |\lambda| \beta^2 +\lambda^2 (d+ \beta(\nu+1) ) + |\lambda| \beta (\gamma+2\nu\beta) 
+ 2\nu \beta^2(\gamma+\nu\beta)\\
&  \leq \lambda^2 \bigl[d+ \beta(\nu+1) + \beta(\beta+ (1+2\nu\beta)(\gamma+2\nu\beta) )\bigr].
\end{align*}
Hence it either holds $|\lambda|\leq 1$, or $|\lambda|>1$ is not larger than the term in the square parenthesis.
\end{proof}
In general, when dealing with equations with one (constant) delay, one is interested in studying stability switches when $\tau$ increases, and looks for the first value $\tau_0$ at which the characteristic equation has a pair of pure imaginary conjugate roots. In our case, the characteristic equation about the endemic equilibrium, $(1/\r0,I^*)$,  
can be written in the form
$$\underbrace{P(\lambda) +Q(\lambda) e^{-\lambda \tau}}_{W(\lambda)}=0,$$
where
\begin{align*}
P(\lambda) & = \lambda^3 + \lambda^2 (d+ \beta I^*)+ \lambda \beta (\gamma+d)I^*+ \nu\beta I^*\mu(\nu)\sigma(\nu),\\[0.2em]
Q(\lambda) & =  \bigl(\lambda^2  \nu  - \lambda  \bigl( \sigma(\nu)  - \nu\beta I^*\bigr)- \nu \beta I^*  \sigma(\nu)\bigr) \mu(\nu),
\end{align*}
and $\mu(\nu),\,\sigma(\nu)$ as in \eqref{def:mu_sigma}. Consider purely imaginary roots, $\lambda = i\omega$, with $\omega>0$, and define $ W(i\omega)=P(i\omega) +Q(i\omega) e^{-i\omega \tau}$. Separate real and imaginary part of $W(i\omega)$, square both terms and add them together. The result is a  quadratic equation,
\begin{equation*}
\xi^2+ a_2\xi +a_1=0,\qquad \xi=\omega^2,
\label{eq:char_PQ_4a}
\end{equation*}
where 
\begin{align*}
a_2 & = \bigl( d^2+ \beta^2 (I^*)^2 - 2\beta\gamma I^* -\nu^2  \mu^2(\nu)\bigr),\\
a_1 & =  (\beta I^*(\gamma+d))^2 -2\nu \beta I^*(d+ \beta I^*)\mu(\nu) \sigma(\nu)\\
& \qquad  -(\sigma(\nu)-\nu\beta I^*)^2\mu^2(\nu)-2 \nu^2  \mu^2(\nu) \beta I^*  \sigma(\nu).
\end{align*}
This equation can have zero, one, or two solutions $\xi$, corresponding to zero, two or four roots $\omega=\pm \sqrt{\xi}$ of $W(i\omega)$. In order to analytically determine stability switches, one usually studies the sign of $d(\Re(\lambda))/d\tau$ at purely imaginary roots $\lambda=i\omega$, or equivalently the sign of 
$\Re( d(\lambda/d\tau)^{-1})$. In our case however, due to the complicated expression \eqref{eq:char_x_5} in which several coefficients, such as $I^*$, $\mu(\nu)$ or $\sigma(\nu)$, depend on the delay, it is not really possible to study $d(\Re(\lambda))/d\tau$ in the general case.

\subsection{Stability with respect to two parameters}
In order to compute regions of stability in a parameter plane, say $(\nu,\tau)$, the classical technique is to separate real and imaginary part of the characteristic equation, substituting $\lambda=x+iy$ in \eqref{eq:char_x_4}, and then obtain an explicit expression for $\nu$ and $\tau$ as a function of the imaginary part $y$. In this way, it is usually possible to have the parametric formulation of curves on which a pair of roots is exactly on the imaginary axis.\\
\ \\
In the case of equation~\eqref{eq:char_x_5}, with $\r0>1$, the characteristic equation at the endemic equilibrium, several coefficients depend on the $I^*$ coordinate of the endemic equilibrium, which in turn depends on all model parameters, including $\nu$ and $\tau$. Therefore it is not possible to solve explicitly for $\nu$ or $\tau$ (or any other parameter). Nevertheless, thanks to the results in the previous sections, we know the stability properties along the $\nu$ and $\tau$ axes. On the one side, if $\tau=0$ (and $\r0>1$) the endemic equilibrium is globally asymptotically stable for all $\nu\geq 0$. On the other side, if $\nu=0$ we know that a number of stability switches occur in $\tau$, and there is a value $\tau_0>0$ at which the first Hopf-bifurcation occurs, so that the endemic equilibrium is locally asymptotically stable for $\tau \in [0,\tau_0)$. We perform a few numerical test with TRACE-DDE~\cite{tracedde} to determine the stability of the endemic equilibrium with respect to two parameters. \\
\ \\
In the following we consider parameter values which are plausible for pertussis disease. Pertussis is a highly transmittable disease ($\r0=15$) with about 21 days infectious period ($\gamma=17$) \cite{CDCpertussis}. Unless otherwise mentioned, we assume turnover an average life time of 50 years ($d=0.02$).\\
\ \\
Figure \ref{Fig1ab}~(a) shows the stability of the endemic equilibrium in the parameter plane ($\nu,\tau$). Notice that the coordinates of the endemic equilibrium change in dependence of the parameters, hence they have to be computed for each parameter pair ($\nu,\tau$). In the above section we have assumed $\nu \in [0,1]$, here for the sake of numerical interest we investigate the stability chart for $\nu\geq 0$. The stable region (green) is the one in which all the characteristic roots have real part smaller than zero. The unstable region (red) indicates parameter combinations for which at least one characteristic root has positive real part. On the curves which separates stable and unstable regions one or more roots are crossing the imaginary axis. We see that increasing $\nu$ has a stabilizing effect on the endemic equilibrium. \\
\ \\
In Fig. \ref{Fig2abcd} we study effects of the mortality rate on the left lower part of the $(\nu,\tau)$ plane. We see that increasing the mortality stabilizes the endemic equilibrium in the sense that the unstable region becomes smaller. This matches previous results on an ODE model with waning immunity and boosting by Dafilis et~al.~\cite{Dafilis2012}.\\
\ \\
Figure~\ref{Fig1ab}~(b) shows the stability of the endemic equilibrium in the parameter plane ($\nu,\r0$), for $\nu \in [0,6]$ and $\r0 \in [1.05,10]$. We observe that for $\nu$ close to $4.5$ there are four stability switches in $\r0$, that is the endemic equilibrium is locally asymptotically stable for value of $\r0$ close to 1, then it becomes unstable, then stable, then unstable and finally again stable. Studying the characteristic roots related to these switches, we find that when the endemic equilibrium loses stability it is due to a Hopf bifurcation (Fig.~\ref{Fig:pertussis_nu_R0_specturm}).\pagebreak

\begin{figure}[h]
\centering
 \includegraphics[width=0.85\textwidth]{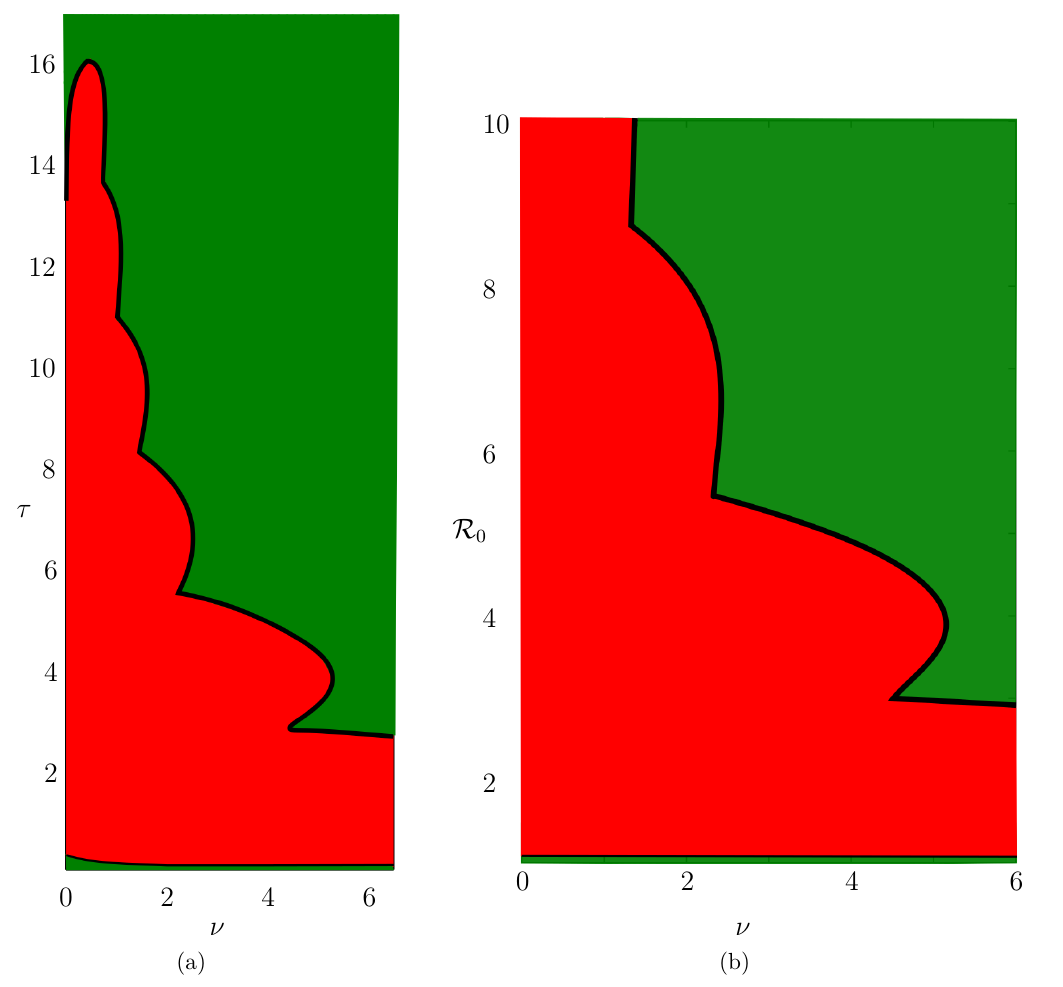}
\caption{Stability of the endemic equilibrium in a two parameter plane. Green region indicates stability, red one instability. (a) Parameter plane ($\nu,\tau$) in the pertussis parameter setting  ($\r0=15$, $\gamma=17$) with population turnover $d=0.02$ (50 years average life time). (b) Parameter plane ($\nu,\r0$) for $\tau=15$, $\gamma=17$ and $d=0.02$.}
 \label{Fig1ab}
\end{figure}

\begin{figure}[h]
\centering
\includegraphics[width=0.85\textwidth]{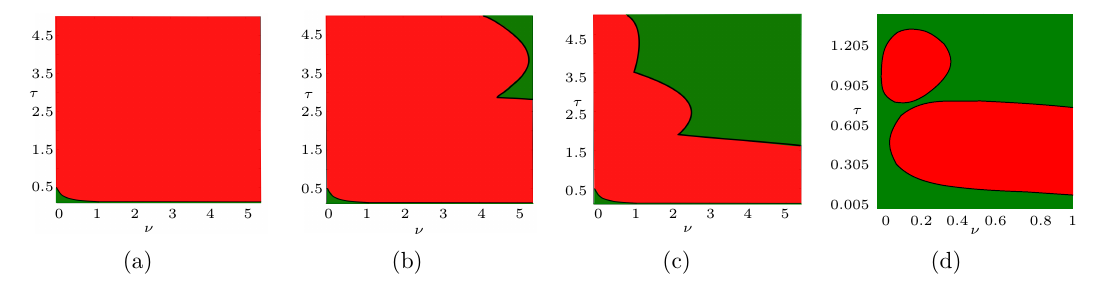}
\caption{Pertussis parameter setting  ($\r0=15$, $\gamma=17$). Effects of the mortality rate on the parameter plane $(\nu,\tau)$. (a) $d=0.013$ corresponding to 75 years, (b) $d=0.02$ corresponding to 50 years, (c) $d=0.05$ corresponding to 20 years, (d) $d=0.2$ corresponding to 5 years average life time. Green region indicates stability, red one instability.}
 \label{Fig2abcd}
\end{figure}

\begin{figure}[h]
\centering
 \includegraphics[width=0.85\textwidth]{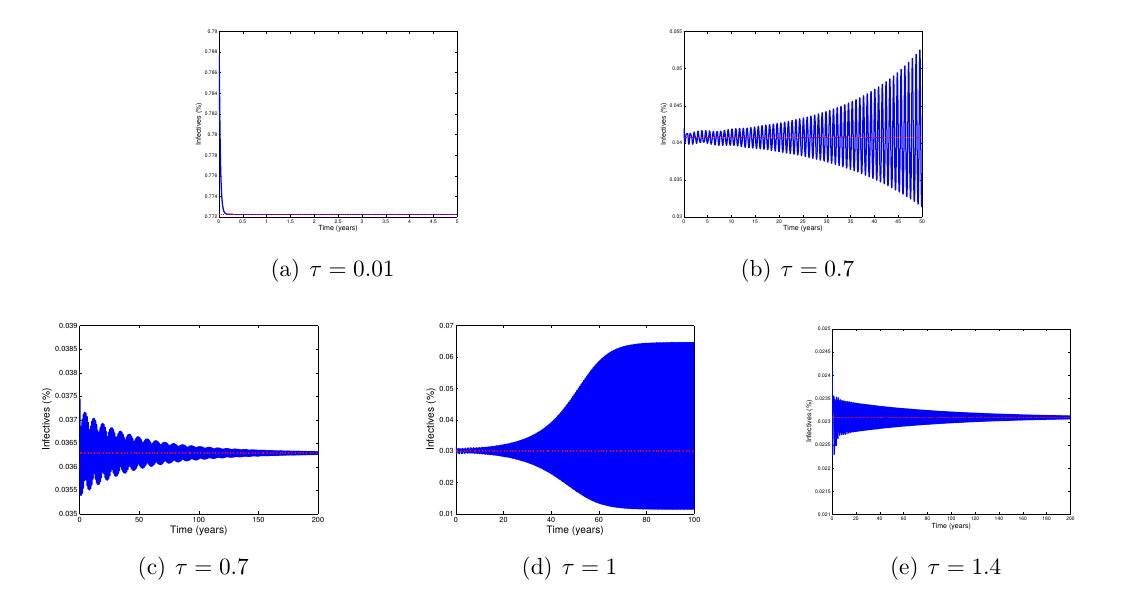}
\caption{Stability switches in dependence of $\tau$. The blue curve shows the infective population $I(t)$, the red dotted line indicates the endemic equilibrium $I^*$. Parameter values correspond to Figure~\ref{Fig2abcd}(d) ($\r0=15$, $\gamma=17$, $d=0.2$) and the delay varies: (a) $\tau=0.01$, (b) $\tau=0.7$, (c) $\tau=0.8$, (d) $\tau=1$ (e) $\tau=1.4$.}
 \label{Fig3abcde}
\end{figure}

\begin{figure}[h]
\centering
 \includegraphics[width=0.85\textwidth]{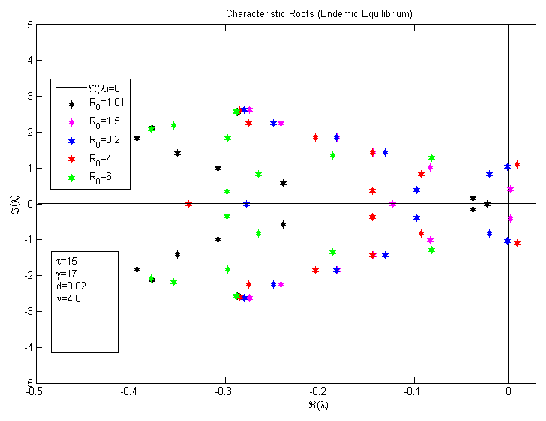}
\caption{Spectrum of the endemic equilibrium when $\nu=4.8$, $d=0.02$, $\tau=15$, $\gamma=17$ and $\r0$ varies. For $\r0=1.01$, $\r0=3.2$ and $\r0=6$ there is no characteristic root on the right half-plane, whereas for $\r0=1.5$ and $\r0=4$ a pair of characteristic roots has crossed the imaginary axis. }
 \label{Fig:pertussis_nu_R0_specturm}
\end{figure}

\noindent We conclude with a short remark on the disease free equilibrium. The characteristic equation at the DFE is 
\begin{equation*}
\lambda^2 \left(\lambda -\beta +d+\gamma \right) = -\lambda d (\lambda + d+\gamma-\beta).
\end{equation*}
As $\lambda=0$ is not a root (Lemma \ref{lemma:lambda0}),  we divide by $\lambda$ left and right, obtaining
\begin{equation*}
\lambda^2 +\lambda\beta\left( d+1/\r0-1 \right) +d\beta(1/\r0-1)=0.
\end{equation*}
Hence the stability of the DFE does not depend on $\nu$, nor on $\tau$, but only on $\r0$. From Theorem~\ref{thm:stabilityDFE} we know that the trivial equilibrium is globally asymptotically stable for $\r0\leq 1$.\pagebreak

\section{Discussion}
 Several models for waning immunity in the form of DDE systems with constant or distributed delay have been proposed in the past few years \cite{Kyrychko2005,Taylor2009,Blyuss2010,Bhat2012,Belair2013}. None of such models, however, includes immune system boosting.\\
 \ \\ 
\noindent In this work we have introduced the model \eqref{sys:newSISdelay} for waning and boosting immunity, written as a system of two differential equations with constant and distributed delay.
The delay $\tau$ represents the duration of immunity after natural infection. One limitation of the proposed model is the assumption that the infectious period is constant. In the future it might be interesting to extend the system including a further distributed delay for variable recovery.\\
\indent For the system~\eqref{sys:newSISdelay} we have proved global existence and uniqueness of solutions on $\Omega \subset C([-\tau,0],[0,1]^2)$. As it often happens in applications of delay differential equations, the solutions of our system can become negative. Non-negative solutions can be obtained by restricting  the choice of possible initial data to an appropriate set $\tilde \Omega \subset \Omega$ in \eqref{eq:tildeOmega} (cf. Theorem~\ref{thm:exiunipos}).\\
\indent Classical analysis of epidemiological models includes the determination of the basic reproduction number $\r0$, a parameter which indicates if and how strongly the disease spreads among the population. For system~\eqref{sys:newSISdelay} $\r0$ is given by the relation~\eqref{def:r0}. We proved that if $\r0$ is larger than one, the system has one unique endemic equilibrium and the disease persists in the population; if $\r0\leq 1$ then the only biologically relevant equilibrium is the DFE and it is globally asymptotically stable. In the limit case $\tau=0$ we get a SIS system with no immunity. The ODE system~\eqref{sys:tau0} has the same reproduction number as \eqref{sys:newSISdelay}. Also in this case the value of $\r0$ determines global stability of the disease free equilibrium (when $\r0\leq 1$) or of the unique endemic equilibrium (when $\r0>1$).\\
\indent We have investigated stability switches of the endemic equilibrium with respect to $\tau>0$, but we have seen that in general it is not possible to study the sign of the characteristic roots with respect to the delay. On the other hand we see in Fig.~\ref{Fig1ab}~(a) and Fig.~\ref{Fig3abcde} that increasing $\tau$, the endemic equilibrium first loses and then regains stability. Our conjecture is that when the delay is large the system~\eqref{sys:newSISdelay} approximates a classical ODE SIR model: immune hosts are protected for a very long time, wiping out the effect of immune system boosting.\pagebreak

\noindent For biological motivation might make sense to consider the boosting force $\nu\in [0,1]$. Observe that if $\nu$ is very large ($\nu\to \infty$), the exponential term in the first equation \eqref{sys:newSISdelay} tends to zero and the dynamics approximates the one of a SIR system without delay. Indeed, numerical simulations in Fig.~\ref{Fig3abcde} show that increasing $\nu$ has a stabilizing effect on the endemic equilibrium. But $\nu$ is not the only parameter which affects the stability region of the nontrivial steady state.
Studying the parameter plane $(\nu,\tau)$, we have found that increasing the mortality rate $d$ stabilizes the endemic equilibrium, in accordance with results in~\cite{Dafilis2012}. For large values of the mortality rate (Fig.~\ref{Fig2abcd}~(d)), the unstable Hopf-bifurcation regions in the $(\nu,\tau)$ plane are (red) spots well separated from each other. Decreasing $d$, each of these spots get larger and different unstable regions might overlap, generating a unique large unstable region with few curves along which a double Hopf-bifurcation occurs (Fig.~\ref{Fig2abcd}~(a-c)). Due to the complicated form of the characteristic equation, including an implicitly determined endemic equilibrium, it was not possible to determine the explicit expression of the Hopf-bifurcation curves in the parameter plane $(\nu,\tau)$ nor in other parameter planes.\\
\indent In the last part of the manuscript, we have constructed a $(\nu,\r0)$ stability chart for the nontrivial steady state. Increasing $\r0$, the endemic equilibrium can experience several stability switchings crossing two distinct regions of instability separated by Hopf-bifurcations. We believe this is a novel bifurcation diagram in epidemic context, which leads the path for further numerical investigations and for careful mathematical analysis.

 \begin{acknowledgements}
MVB is supported by the European Social Fund and by the Ministry of Science, Research and Arts Baden-W\"urttemberg. MP is supported by the Hungarian Scientific Research Fund, Grant No. K109782. GR is supported by the ERC Starting Grant No. 259559.
 \end{acknowledgements}

\end{document}